\documentclass[10pt,bezier]{article}
\usepackage{amsmath,amssymb,amsfonts,graphicx,caption,subcaption}
\usepackage[colorlinks,linkcolor=red,citecolor=blue]{hyperref}

\newtheorem{prethm}{{\bf Theorem}}[section]
\newenvironment{thm}{\begin{prethm}{\hspace{-0.5
em}{\bf.}}}{\end{prethm}}
\newtheorem{prepro}{{\bf Theorem}}

\newtheorem{precor}[prethm]{{\bf Corollary}}
\newenvironment{cor}{\begin{precor}{\hspace{-0.5
em}{\bf.}}}{\end{precor}}
\newtheorem{preconj}[prethm]{{\bf Conjecture}}

\newtheorem{preremark}[prethm]{{\bf Remark}}

\newtheorem{prelem}[prethm]{{\bf Lemma}}
\newenvironment{lem}{\begin{prelem}{\hspace{-0.5
em}{\bf.}}}{\end{prelem}}
\newtheorem{preque}[prethm]{{\bf Question}}

\newtheorem{preobserv}[prethm]{{\bf Observation}}

\newtheorem{predef}[prethm]{{\bf Definition}}

\newtheorem{preproposition}[prethm]{{\bf Proposition}}

\newtheorem{preproof}{{\bf Proof.}}
\newtheorem{preprooff}{{\bf Proof}}

\newenvironment{proof}[1]{\begin{preproof}{\rm
#1}\hfill{$\Box$}}{\end{preproof}}

\newtheorem{preproofs}{{\bf The second proof of }}

\newtheorem{preprooft}{{\bf Third proof of }}

\newtheorem{preproofF}{{\bf Proof of}}

\title{\bf\Large 
The existence of planar $4$-connected essentially $6$-edge-connected graphs with no claw-decompositions
}
\author{{\normalsize{\sc Morteza Hasanvand${}$} }\vspace{3mm}
\\{\footnotesize{${}$\it Department of Mathematical
 Sciences, Sharif
University of Technology, Tehran, Iran}}
{\footnotesize{}}\\{\footnotesize{ $\mathsf{morteza.hasanvand@alum.sharif.edu }$ }}}

\date{}

\begin{document}
\maketitle
\begin{abstract}{
In 2006 Bar{\'a}t and Thomassen conjectured that every planar $4$-edge-connected $4$-regular simple graph of size divisible by three admits a claw-decomposition. Later, Lai (2007) disproved this conjecture by a family of planar graphs with edge-connectivity $4$ which the smallest one contains $24$ vertices. In this note, we first give a smaller counterexample having only $18$ vertices and next construct a family of planar $4$-connected essentially $6$-edge-connected $4$-regular simple graphs of size divisible by three with no claw-decompositions. This result provides the sharpness for two known results which say that every $5$-edge-connected graph of size divisible by three admits a claw-decomposition if it is essentially $6$-edge-connected or planar.
} {\small
}
\end{abstract}
%
%
%
%
%
%
%
%
%
%
\section{Introduction}
In this article, graphs have no loops, but multiple edges are allowed, and a simple graph 
have neither loops nor multiple edges.
Let $G$ be a graph. The vertex set, the edge set, and the maximum degree of vertices of $G$ are denoted by $V(G)$, $E(G)$, and $\Delta(G)$, respectively.
We denote by $d_G(v)$, $d^-_G(v)$, and $d^+_G(v)$, the degree, the in-degree, and the out-degree of a vertex $v$ in the graph $G$. For a vertex set $A$, we denote by $e_G(A)$ the number of edges with both ends in $A$.
An orientation of the graph $G$ is said to be 
{\it $p$-orientation}, if for each vertex $v$, $d_G^+(v)\stackrel{k}{\equiv}p(v) $,
 where $p:V(G)\rightarrow Z_k$ is a mapping and $Z_k$ is the cyclic group of order $k$.
For the zero function $p$, the graph $G$ has a $p$-orientation if and only if it admits a $k$-star-decomposition.
A graph is termed {\it essentially $\lambda$-edge-connected},
 if the edges of any edge cut of size strictly less than $\lambda$ are incident with a common vertex.

In 2006 Bar{\'a}t and Thomassen~\cite{Barat-Thomassen-2006}
 conjectured that every planar $4$-edge-connected $4$-regular simple graph $G$ of size divisible by $3$ admits a claw-decomposition.
 Later, Lai (2007)~\cite{Lai-2007} disproved this conjecture by a class of planar graphs with vertex-connectivity two.
Former, Lai and Li (2006) \cite{Lai-Li-2006} proved the following stronger assertion but for planar $5$-edge-connected graphs, in terms of $Z_3$-connectivity, 
using the duality of planar graphs with graph colouring. A directive proof of Theorem~\ref{intro:thm:planar} is found by Richter, Thomassen, and Younger (2016)~\cite{Richter-Thomassen-Younger-2016} and this theorem is recently developed to the projective planar graphs by Jong and Bruce~(2020)~\cite{Jong-Richter}. 
\begin{thm}{\rm(\cite{Lai-Li-2006})}\label{intro:thm:planar}
{Let $G$ be a planar graph and let $p:V(G)\rightarrow Z_3$ be a mapping with
$|E(G)| \stackrel{3}{\equiv} \sum_{v\in V(G)}p(v)$.
 If $G$ is $5$-edge-connected, then
it has a $p$-orientation.
}\end{thm}
\begin{cor}{\rm (see Theorem 4.2 in~\cite{Richter-Thomassen-Younger-2016}}\label{cor:claw}
{Every $5$-edge-connected planar graph of size divisible by $3$ admits a claw-decomposition.
}\end{cor}

In 2012 Thomassen \cite{Thomassen-2012} developed Theorem~\ref{intro:thm:planar} to $8$-edge-connected graphs and succeeded to confirm another beautiful conjecture proposed by Bar{\'a}t and Thomassen~\cite{Barat-Thomassen-2006} about the existence of claw-decomposition in graphs with high enough edge-connectivity. Next, Lov{\'a}sz, Thomassen, Wu, and Zhang (2013) refined Thomasen's result by replacing the needed edge-connectivity $6$.
In particular, they proved a more stronger version which contains the following result as a corollary on
essentially edge-connected graphs.
\begin{thm}{\rm (\cite{Lovasz-Thomassen-Wu-Zhang-2013}, see Theorem 1.1 in \cite{Delcourt-Postle-2018})}\label{intr:thm:claw}
{Every $5$-edge-connected essentially $6$-edge-connected graph of size divisible by $3$ admits a claw-decomposition.
}\end{thm}

In this note, we show that Bar{\'a}t and Thomassen's conjecture~\cite{Barat-Thomassen-2006} does not hold 
 in planar $4$-connected essentially $6$-edge-connected simple graphs by giving a new family of such $4$-regular graphs with no claw-decompositions. This shows that (i) the needed edge-connectivity in Corollary~\ref{cor:claw} is best possible even for
 essentially $6$-edge-connected graphs and (ii) the needed edge-connectivity in Theorem~\ref{intr:thm:claw} is best possible even in planar graphs. In Section~\ref{sec:star}, we also give another family of graphs with high essential edge-connectivity but with no 
$k$-star-decompositions, and provide a useful criterion for the existence of $k$-star-decompositions in graphs with maximum degree at most $2k-1$; in particular, the existence of claw-decompositions in $4$-regular graphs. 
%
%
%
%
%
%
%
%
%
%
%
%
%
\section{Claw-decompositions in $4$-regular graphs}
In 1992 Jaeger, Linial, Payan, and Tarsi~\cite{Jaeger-Linial-Payan-Tarsi-1992} constructed a $4$-edge-connected $4$-regular simple graph of order $12$ with no claw-decompositions (the third graph in Figure~\ref{Figure:12}). 
By a computer search, we observe that there are only four $4$-regular connected simple graphs of order $12$ with no claw-decompositions using a regular generator due to Meringer (1999)~\cite{Meringer-1999}. 
In addition, we observe that there are only $146$ graphs without claw-decompositions among $4$-regular connected simple graphs of order $15$ which is less than $0.02\%$ of them. Also, there are only $15932$ such graphs among $4$-regular connected simple graphs of order $18$ which is less than $0.002\%$ of them. Recently, Delcourt and Postle (2018)~\cite{Delcourt-Postle-2018} proved that this ratio must tend to zero.
\begin{figure}[h]
{\centering
\includegraphics[scale=1.15]{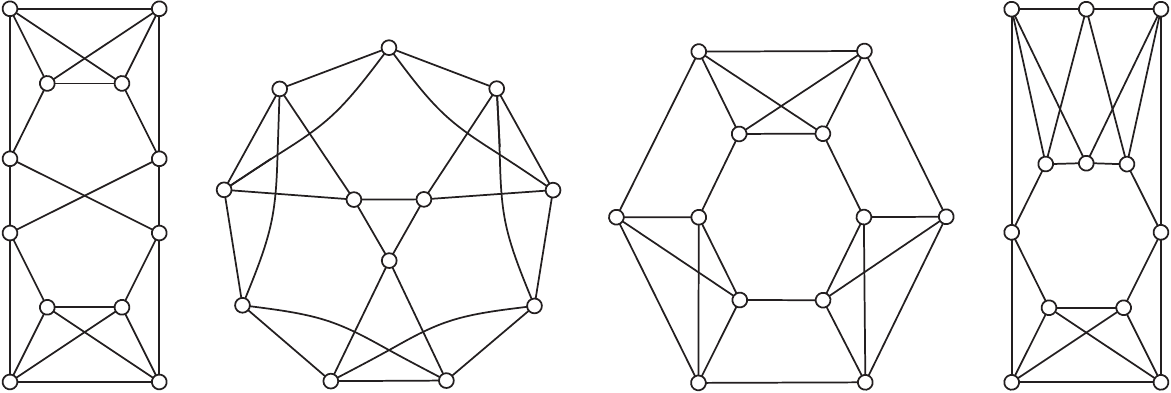}
\caption{All $4$-regular connected simple graphs of order $12$ with no claw-decompositions.}
\label{Figure:12}
}\end{figure}

For planar graphs, Lai (2007)~\cite{Lai-2007} introduced a family of planar $2$-connected $4$-edge-connected $4$-regular simple graphs of order $12$ with no claw-decompositions which the smallest one contains $24$ vertices. By a computer search (using a planar graph generator due to Brinkmann and McKay~\cite{Brinkmann-McKay-2007}), we observe the there is a smaller such planar graph containing only $18$ vertices that illustrated in Figure~\ref{Figure:18}.
\begin{figure}[h]
{\centering
\includegraphics[scale=.77]{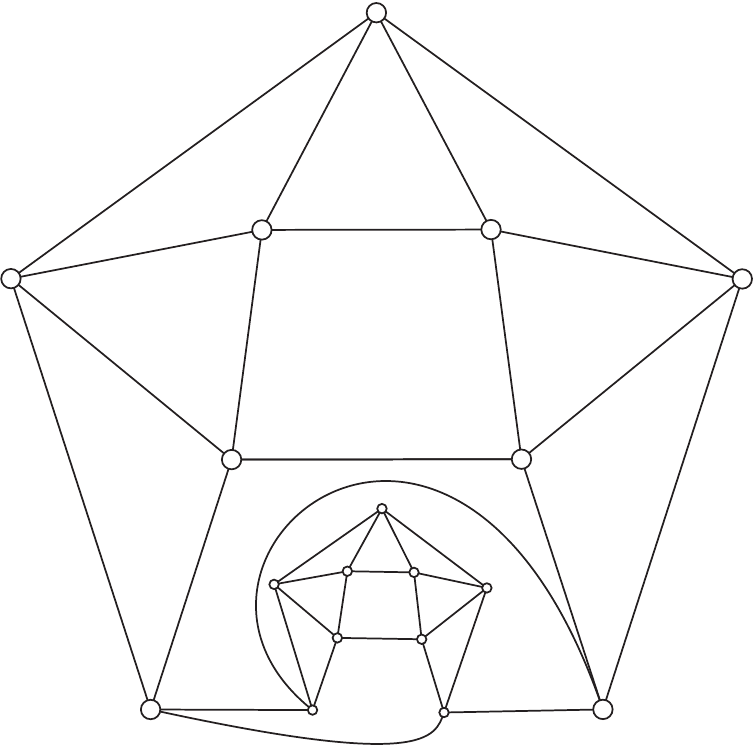}
\caption{A planar $2$-connected $4$-regular graph of order $18$ with no claw-decompositions.}
\label{Figure:18}
}\end{figure}

Motivated by Theorem~\ref{intr:thm:claw}, one may ask whether there is such a planar graph with higher essential edge-connectivity or even vertex-connectivity. By searching among planar $3$-connected $4$-regular graphs of order $21$,
 we discover a number of such desired planar graphs. Among them, some ones meet vertex-connectivity $4$ and some ones meet essential edge-connectivity $6$; for example, see Figure~\ref{Figure:18}.
According to Corollary~\ref{cor:claw:criterion}, one can easily prove these graphs do not have a claw-decomposition 
using independent sets.
For instance, the graph in Figure~\ref{Figure:18} its independence number is $5$, 
the right graph in Figure~\ref{Figure:21} its independence number is $6$, 
and the left graph Figure~\ref{Figure:21} has a unique independent set of size $7$ (up to isomorphism) such that by removing it the resulting graph has a component with two cycles.
\begin{figure}[h]
{\centering
\includegraphics[scale=1.42]{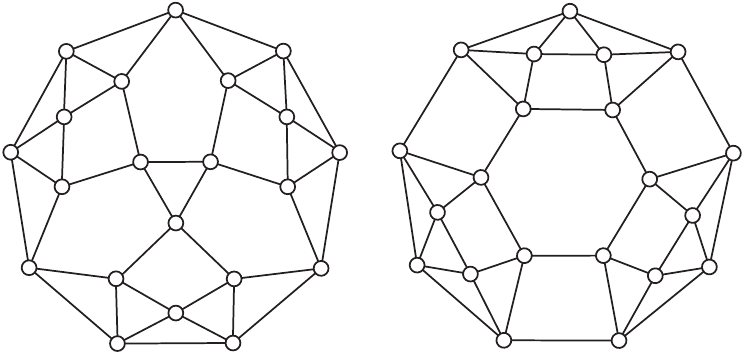}
\caption{Two planar $3$-connected $4$-regular graphs of orders $21$ with no claw-decompositions.}
\label{Figure:21}
}\end{figure}

It remains to decide whether Theorem~\ref{intro:thm:planar} holds for planar $4$-connected essentially $6$-edge-connected graphs, except for a finite number of graphs. We show that the answer is surprisingly false by the following graph construction.
\begin{thm}\label{thm:main}
{There are infinitely many 
planar $4$-connected essentially $6$-edge-connected $4$-regular simple graphs
 of size divisible by $3$ 
 with no claw-decompositions.
}\end{thm}
\begin{proof}
{Consider $3n$ copies of the graph in Figure~\ref{Figure:Block} 
\begin{figure}[h]
{\centering
\includegraphics[scale=1.05]{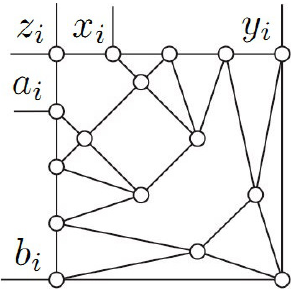}
\caption{The block for constructing the family of graphs $G_{48n}$.}
\label{Figure:Block}
}\end{figure}
and for every $i\in Z_{3n}$, add three edges 
$z_iz_{i+1}$,
 $x_{i}a_{i+1}$, and
$y_ib_{i+1}$ to the new graph.
Call the resulting graph $G_{48n}$ which has $48n$ vertices.
 Figure~\ref{Figure:48} illustrates the graph $G_{48}$ in which each $z_i$ lies in the outer-face.
As observed in \cite{Barat-Thomassen-2006, Lai-2007},
 if a $4$-regular graph $G$ has a claw-decomposition, 
then the non-center vertices form an independent set of size $|V(G)|/3$.
If $G_{48n}$ has a claw-decomposition, then it must have an independent set $X$ of size $16n$.
But $X$ includes at most $5$ vertices from every block and hence it has at most $15n$ vertices which is a contradiction.
The vertex connectivity and essentially edge-connectivity of $G_{48n}$ can easily be verified.
The proof is left to the reader.
}\end{proof}
\begin{figure}[h]
{\centering
\includegraphics[scale=1.5]{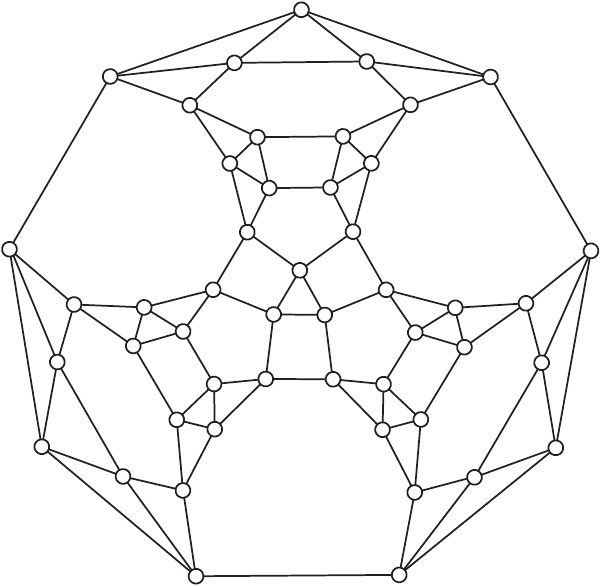}
\caption{A planar $4$-connected essentially $6$-edge-connected $4$-regular graph of order $48$ 
 with no claw-decompositions.}
\label{Figure:48}
}\end{figure}
\section{Graphs with high essential edge-connectivity and without $k$-satr-decompositions}
\label{sec:star}
It is known that every $(2k-1)$-edge-connected essentially $(3k-3)$-edge-connected graph $G$ of size is divisible by $k$ with $k\ge 3$ admits a $k$-star-decomposition, and there are $(2k-2)$-edge-connected $(2k-2)$-regular graphs of size divisible by $k$ with no $k$-star-decompositions, see \cite{Delcourt-Postle-2018, ModuloBounded, Lovasz-Thomassen-Wu-Zhang-2013}.
Motivated by Theorem~\ref{thm:main}, we are going to show that there are such regular graphs
 with the highest essential edge-connectivity but without $k$-star-decompositions.
\begin{thm}
{Fot any integer $k$ with $k\ge 3$, there are infinitely many 
$(2k-2)$-connected essentially $(4k-6)$-edge-connected $(2k-2)$-regular simple graphs of size divisible by $k$ 
 with no $k$-star-decompositions.
}\end{thm}
\begin{proof}
{We may assume that $k\ge 4$ as the assertion holds by 
Theorem~\ref{thm:main} for the special case $k=3$. 
 Take $G$ to be the Cartesian product of the cycle of order $kn$ and the complete graph of order $2k-3$, where $n$ is an arbitrary positive integer. It is not hard to check that $G$ is $(2k-2)$-connected essentially $(4k-6)$-edge-connected 
$(2k-2)$-regular simple graph of size divisible by $k$. We claim that $G$ has no $k$-star-decompositions.
Otherwise, the number of stars must be $(1-1/k)|V(G)|$, since $G$ contains $(k-1)|V(G)|$ edges.
Thus the number of non-center vertices must be $|V(G)|/k$ and these vertices form an independent set of $G$.
On the other hand, according to the construction, the graph $G$ whose independence number is at most $|V(G)|/(2k-3)$. Since $2k-3 > k$, we arrive at a contradiction.
}\end{proof}
In the following, we are going to present a helpful criterion for the existence of $k$-star-decompositions in terms of independent sets. For this purpose, we need the following well-known theorem due to Hakimi~(1965).
\begin{lem}{\rm (\cite{Hakimi-1965})}\label{lem:Hakimi}
{Let $G$ be a graph and let $p$ be an integer-valued function on $V(G)$.
Then $G$ has an orientation such that for all $v\in V(G)$, $d^-_G(v)\le p(v)$, if and only if for all $S\subseteq V(G)$,
$$ e_G(S)\le \sum_{v\in S}p(v),$$ 
}\end{lem}
Now, are ready to prove the next assertion.
\begin{thm}\label{thm:criterion}
{Let $G$ be a graph of size divisible by $k$ satisfying $\Delta(G)\le 2k-1$ which $k$ is an integer number with $k\ge 3$. 
Then $G$ admits a $k$-star-decomposition if and only if it has an independent set $S$ of size $|V(G)|-\frac{1}{k}|E(G)|$ such that for every $A\subseteq V(G)\setminus S$, $$e_G(A)\le \sum_{v\in V(G)\setminus S}(d_G(v)-k).$$
}\end{thm}
\begin{proof}
{First assume that there is an independent set $S$ of size $|V(G)|-|E(G)|/k$ satisfying the theorem. 
By Lemma~\ref{lem:Hakimi}, there is an orientation for $G\setminus S$ such that every vertex of it 
has in-degree at most $d_G(v)-k$. 
According to the assumption, we also have 
$$\sum_{v\in V(G)\setminus S}k=k(|V(G)|-|S|)=|E(G)|=\sum_{v\in V(G)\setminus S}d_G(v)-|E(G\setminus S)|.$$
Therefore, $G\setminus S$ its size must be $\sum_{v\in V(G)\setminus S}(d_G(v)-k)$ and hence every vertex of it has in-degree $d_G(v)-k$. Let us orient the remaining edges from $V(G)\setminus S$ to $S$ to obtain an orientation for $G$
 so that every vertex in $S$ has out-degree zero and every vertex in $V(G)\setminus S$ has out-degree $k$.
Obviously, this orientation induces a $k$-star-decomposition for $G$.

Now, assume that $G$ has a $k$-star-decomposition.
 Obviously, the number of stars must be $|E(G)|/k$. Since $G$ has maximum degree at most $2k-1$, every vertex is the center of at most one star. Thus the number of center vertices must be $|E(G)|/k$. 
If we set $S$ to be the set of all non-center vertices, then this set must be an independent set and we must have $|S|=|V(G)|-|E(G)|/k$. 
Let us orient the edges of $G$ such that the edges of every star directed away from the center. 
This implies that every vertex $V(G)\setminus S$ has in-degree at most $d_G(v)-k$ in $G$ and so does in $G\setminus S$. By Lemma~\ref{lem:Hakimi}, for every $A\subseteq V(G)\setminus S$, $e_G(A)\le \sum_{v\in V(G)\setminus S}(d_G(v)-k)$.
Hence the proof is completed.
}\end{proof}
The following corollary is a useful tool to show that why the left graph in Figure~\ref{Figure:21} does not have a claw-decomposition.
More precisely, it has a unique independence set of size $|V(G)|/3$ (up to isomorphism).
\begin{cor}\label{cor:claw:criterion}
{Let $G$ be a $4$-regular graph of size divisible by three. Then $G$ admits a claw-decomposition if and only if it has an independent set $S$ of size $|V(G)|/3$ such that every component of $V(G)\setminus S$ contains exactly one cycle.
}\end{cor}
\begin{proof}
{Apply Theorem~\ref{thm:criterion} and use the fact that a graph $H$ of size $|V(H)|$ satisfies $e_H(A)\le |S|$ for every $A\subseteq V(H)$, if and only if every component of it contains exactly one cycle. Note that if $S$ is an independent set of size $V(G)/3$, then the number of edges of $G\setminus S$ must be $|V(G)\setminus S|$. 
}\end{proof}
%
%
%
%
%
%
%
%
%
%
%

\end{document}